\documentclass{article}
\usepackage{graphicx} 

\usepackage[top=2.54cm, bottom=2.54cm, left=3.17cm, right=3.17cm]{geometry}
\usepackage{amsmath}    
\usepackage{amsfonts}    
\usepackage{amsthm} 
\usepackage{mathtools} 
\usepackage{amssymb} 
\usepackage{enumitem} 
\usepackage{tikz} 
\usepackage{subcaption} 
\usepackage{indentfirst} 
\usepackage{hyperref} 

\usepackage[dvipsnames]{xcolor}

\tikzset{every picture/.style={line width=0.75pt}}

\newtheorem{theorem}{Theorem}
\newtheorem{lemma}[theorem]{Lemma}

\newtheorem{claim}[theorem]{Claim}

\newtheorem{corollary}[theorem]{Corollary}
\newtheorem{problem}[theorem]{Problem}
\theoremstyle{remark}
\newtheorem{remark}[theorem]{Remark}
\newtheorem*{remark*}{Remark}
\newtheorem{example}[theorem]{Example}

\title{New Helly-type results for discrete boxes:\\ Quantitative colorful and $(p,q)$-variants}
\author{Rahul Gangopadhyay\thanks{Department of Computer Science and Engineering, Indian Institute of Technology (ISM), Dhanbad, India. Email address: \href{mailto:rahulgangopadhyay@iitism.ac.in}{rahulgangopadhyay@iitism.ac.in}}  
    \and
    Alexander Polyanskii\thanks{Emory University, 400 Dowman Drive, 30322, Atlanta GA, USA. Email address: \href{mailto:apolian@emory.edu}{apolian@emory.edu}. Homepage: \url{http://polyanskii.com/}.}
    \and
    Wei Rao\thanks{Moscow Institute of Physics and Technology, Institutsky lane 9, Dolgoprudny, Moscow region, 141700, Russia. Email address: \href{mailto:raowei1998@gmail.com}{raowei1998@gmail.com}}    
}

\date{}

\begin{document}

\maketitle

\begin{abstract}
In 2008, Halman showed that for any finite set $P\subset \mathbb R^d$ and any finite family \( \mathcal{B} \) of axis-parallel boxes in \( \mathbb{R}^d \), if the intersection of \( P \) and any subfamily \( \mathcal{B}' \subseteq\mathcal{B} \) of size at most \( 2d \) is non-empty, then the intersection of \( P \) and  \( \mathcal{B} \) is also non-empty. Very recently Edwards and Sober\'on initiated the study of quantitative colorful version for $2d$ families, $(p,q)$-type variation for $p\geq q\geq d+1$, and other extensions of this Helly-type result by Halman. 

In this paper, we study the quantitative colorful Halman problem for $2d-1$ families as well its $(p,q)$-type variation for $p\geq q\geq 2$. Specifically, our main result asserts that for any finite set $P$ and finite families of boxes $\mathcal{B}_1,\dots,\mathcal{B}_{2d-1}$ in $\mathbb R^d$, where $d\geq 2$, if every transversal $\mathcal{B}$ for the families has an intersection $\bigcap \mathcal{B}$ containing at least $n$ points of $P$, then there exist $j\in[2d-1]$ and a subset of $P$ of size at most
\[
2n+\Big\lfloor \frac{n-1}{d \cdot 2^{d-1}} \Big\rfloor,
\]
such that each box of $\mathcal{B}_j$ contains at least $n$ points of this subset.

\end{abstract}

{\bf Keywords:} colorful Helly's theorem

\section{Introduction}

In 1923, Helly~\cite{helly1923mengen} published one of the seminal results in discrete geometry. A family $\mathcal F$ of sets is said to be \textit{intersecting} if its intersection $\bigcap \mathcal F$ is non-empty. The Helly theorem states that for a finite family $\mathcal K$ of convex sets in $\mathbb R^d$, if any of its subfamilies of at most $d+1$ sets is intersecting, then \( \mathcal K \) itself is intersecting. This theorem has various classical extensions, including its colorful and $(p,q)$-type versions.

According to B\'ar\'any~\cite{barany1982generalization}, Lov\'asz proved the so-called colorful Helly theorem. For families $\mathcal F_1,\dots, \mathcal F_n$ of sets, a family $\mathcal F'$ is called a \textit{transversal} if $\mathcal F'=\{F_1,\dots, F_n\}$ for some $F_1\in \mathcal F_1, \dots, F_n\in \mathcal F_n$. The families $\mathcal F_1,\dots, \mathcal F_n$ are said to have the \textit{colorful Helly property} if any of their transversals is intersecting. Lov\'asz showed that if finite families $\mathcal K_1,\dots, \mathcal K_{d+1}$ of convex sets in $\mathbb R^d$ have the colorful Helly property, then one of the families $\mathcal K_i$ is intersecting.

In 1992, Alon and Kleitman~\cite{alon1992piercing} proved the so-called $(p,q)$ theorem. For positive integers $p\geq q$, a family $\mathcal F$ of sets has the \textit{$(p,q)$ property} if for any subfamily $\mathcal G$ of $\mathcal F$ of $p$ sets, there is an intersecting subfamily of $\mathcal G$ of size $q$. For a family $\mathcal F$, its \textit{piercing} (or \textit{transversal}) \textit{number}, denoted by $\tau(\mathcal F)$, is the smallest size of a set $S$ intersecting each member of $\mathcal F$. In particular, a family $\mathcal F$ is intersecting if and only if $\tau (\mathcal F)=1$. Alon and Kleitman showed that for any $p\geq q\geq d+1$, there exists $n:=n(p,q,d)$ such that if a finite family $\mathcal K$ of convex sets in $\mathbb R^d$ has the $(p,q)$ property, then $\tau(\mathcal K)\leq n$.
\smallskip

In 2008, Halman~\cite{Hal} studied intersection properties of the trace of a family of boxes on a finite point set. An \textit{axis-parallel box} (or simply a \textit{box}) in $\mathbb R^d$ is the Cartesian product of $d$ line segments of $\mathbb R$. (In this paper, by a line segment we mean a \textit{closed} line segment, unless stated otherwise.)
For a set $X$ and a family $\mathcal F$ of sets,  the \textit{trace} of $\mathcal F$ on $X$, denoted by  $\mathcal F|_X$, is the family 
\(
\{F\cap X:F\in \mathcal F\}.
\)

\begin{theorem}[Halman theorem; Theorem 2.10 in \cite{Hal}]
\label{Halman theorem}
    Let $d$ be a positive integer. Let $P$ be a finite set in $\mathbb R^d$ and let $\mathcal B$ be a finite family of boxes in $\mathbb R^d$. If for any subfamily $\mathcal B'\subseteq \mathcal B$ of size at most \(2d\), its trace $\mathcal B'|_P$ is intersecting, then \(\mathcal{B}|_P\) is also intersecting.
\end{theorem}

Recently, Edwards and Sober{\'o}n~\cite{edwards2025extensions} initiated studies of generalizations of Halman's result, including a quantitative colorful extension on $n$-intersecting families and a $(p,q)$-type version. A family $\mathcal F$ of sets is said to be \textit{$n$-intersecting} if $|\bigcap \mathcal F|\geq n$.

\begin{theorem}[quantitative colorful Halman theorem for $2d$ colors, Theorem 1.2 in \cite{edwards2025extensions}]\label{quantitative colorful Halman theorem for 2d colors}
    Let $d$ and $n$ be positive integers. Let $P$ be a finite set in $\mathbb R^d$ and let $\mathcal{B}_1, \dots, \mathcal{B}_{2d}$ be finite families of boxes in $\mathbb R^d$. If for any transversal $\mathcal B$, its trace $\mathcal B|_P$  is $n$-intersecting, then there is $i\in[2d]$ such that the trace $\mathcal B_i|_P$ is $n$-intersecting.
\end{theorem}

\begin{theorem}[$(p,q)$-type Halman theorem for $p\geq q\geq d+1$; Theorem 1.4 in \cite{edwards2025extensions}]\label{pq Halman theorem for q at least d plus 1}
  Let $p,q,$ and $d$ be positive integers with $p\geq q\geq d+1$. There exists $N:=N(p,q,d)$ such that any finite set $P\subset \mathbb R^d$ and any finite family $\mathcal B$ of boxes in $\mathbb R^d$ satisfy the following. If the trace $\mathcal B|_P$ has the $(p,q)$ property and does not contain the empty set, then $\tau(\mathcal B|_P)\leq N$.
\end{theorem}

Jensen, Joshi, and Ray~\cite{Jen} asked to find minimal $m$ such that any finite set $P\subset \mathbb R^2$ and any finite family $\mathcal B$ of boxes satisfy the following Helly-type property. If for any $\mathcal B'\subseteq \mathcal B$ of size 3, its trace $\mathcal B'|_P$ is intersecting, then $\tau(\mathcal B|_P)\leq m$. To the best of our knowledge no paper has addressed this question. It is worth mentioning that Theorem~\ref{monochromatic version} and Example~\ref{example lower bound for monochromatic version} imply that $m=2$.

The goal of the current paper is to prove further quantitative colorful and $(p,q)$-type extensions of Halman theorem. In particular, we answer to the question of Jensen, Joshi, and Ray.

\medskip

Motivated by the quantitative colorful Halman theorem for \(2d\) colors by Edwards and Sober\'on, we examine the corresponding problem for \(2d - 1\) colors. To state our main result, we introduce a quantitative generalization of the transversal number. For a family \(\mathcal{F}\) of sets, a \textit{$k$-multitransversal} is a set \(S\) such that each member of \(\mathcal{F}\) contains at least \(k\) elements of \(S\). The \textit{$k$-multitransversal number} of \(\mathcal{F}\), denoted by \(\tau_k'(\mathcal{F})\), is the smallest size of a $k$-multitransversal.\footnote{F{\"u}redi \cite{furedi1988matchings} introduced the notion of $k$-covers to interpolate between the transversal number and the fractional transversal number. A $k$-cover is closely related to a $k$-multitransversal, except that $S$ in the definition is allowed to be a multiset rather than a set. Specifically, for a family \(\mathcal{F}\) of sets, a \textit{$k$-cover} is a multiset \(S\) such that each member of \(\mathcal{F}\) contains at least \(k\) elements of \(S\). The \textit{$k$-covering number} of \(\mathcal{F}\), denoted by \(\tau_k(\mathcal{F})\), is the minimum size of a $k$-cover.}
 In particular, \(\tau_1'(\mathcal{F}) = \tau(\mathcal{F})\).

\begin{theorem}[quantitative colorful Halman theorem for $2d-1$ colors]\label{quantitative colorful Halman 2d-1 colors}
Let \( n \geq 1 \) and \( d \geq 2 \) be integers. Let $P\subset \mathbb R^d$ be a finite set, and let \( \mathcal{B}_1, \dots, \mathcal{B}_{2d-1} \) be finite families of boxes in $\mathbb R^d$. If for any their transversal $\mathcal B$, its trace $\mathcal B|_P$ is \( n \)-intersecting, then there exists $j\in [2d-1]$ such that   \[
\tau_n'(\mathcal B_i|_P)\leq 2n+\Big\lfloor \frac{n-1}{d\cdot 2^{d-1}} \Big\rfloor.
\]
\end{theorem}

We provide an example showing that this bound is close to being tight: For any \( n \geq 1 \) and \( d \geq 2 \), there exists a finite set \( P \subset \mathbb{R}^d \) and finite families \( \mathcal{B}_1, \dots, \mathcal{B}_{2d-1} \) of boxes in \( \mathbb{R}^d \) such that for any transversal \( \mathcal{B} \) of them, its trace \( \mathcal{B}|_P \) is \( n \)-intersecting and satisfies \( \tau_n'(\mathcal{B}_j|_P) = 2n \) for any \( j \in [2d-1] \). The obtained example is of independent interest.

\smallskip 

Another result is the following \((p,q)\)-type theorem for boxes, which extends the $(p,q)$-theorem by Edwards and Sober\'on from the range \(p \geq q \geq d+1\) to \(p \geq q \geq 2\).

\begin{theorem}[$(p,q)$-type Halman theorem for $p\geq q\geq 2$]
\label{pq Halman theorem for q at least 2}
  Let $p,q,$ and $d$ be positive integers with $p\geq q\geq 2$. There exists $N:=N(p,q,d)$ such that any finite set $P\subset \mathbb R^d$ and any finite family $\mathcal B$ of boxes in $\mathbb R^d$ satisfy the following. If the trace $\mathcal B|_P$ has the $(p,q)$ property and does not contain the empty set, then $\tau(\mathcal B|_P)\leq N$.
\end{theorem}

\smallskip

The paper is organized as follows. In Section~\ref{section quantitative Halman theorems}, we address our quantitative Halman results. In particular, Section~\ref{subsection proof of quantitative colorful Halman theorem 2d-1 colors}  contains the proof of Theorem~\ref{quantitative colorful Halman 2d-1 colors} and the example proving a lower bound for the quantitative colorful Halman problem for $2d-1$ families, and then in Section~\ref{section: monochromatic version}, we prove upper and lower bounds for the quantitative monochromatic Halman problem for subfamilies of size $2d-1$ --- this answers the question of Jensen, Joshi, and Ray. In Section~\ref{section pq Halman theorem}, we prove Theorem~\ref{pq Halman theorem for q at least 2}, show a non-trivial bound $N(2,2,2)\geq 3$, and provide an upper bound for the monochromatic Halman problem for subfamilies of size $2d-k$, where $1\leq k\leq d-1$,  as a corollary of Theorem~\ref{pq Halman theorem for q at least 2}. Finally, in Section~\ref{section: discussion}, we propose some related problems.

\section{Quantitative Halman theorems}
\label{section quantitative Halman theorems}
\subsection{Proof of Theorem~\ref{quantitative colorful Halman 2d-1 colors}}
\label{subsection proof of quantitative colorful Halman theorem 2d-1 colors} 
\subsubsection{Notation and Lemma~\ref{claim induction reduction}}
\label{section notation}
For convenience, let \( I := [d] \) denote the set of indices corresponding to the coordinate axes of \( \mathbb{R}^d \). Similarly, let \( J := [2d-1] \) denote the set of indices corresponding to the families \( \mathcal{B}_j \) of boxes. For any \(i \in I\), let \(\pi_i : \mathbb{R}^d \to \mathbb{R}\) denote the projection onto the \(i\)th coordinate axis, which we refer to as the $i$-projection for brevity.

Using a standard compactness argument, without loss of generality, we assume that for each \( i \in I \), all projections \( \pi_i(p) \) for \( p \in P \), as well as the endpoints of the \( i \)-projections (which are segments) of boxes from the families \( \mathcal{B}_j \), are pairwise distinct.

In the proof below, we repeatedly use the following property: 
\begin{equation}
\label{equation colorful Helly for R1}
    \text{for any $i\in I$, there is at most one $j\in J$ with $\bigcap \pi_i(\mathcal B_j)=\emptyset$.}
\end{equation} 
Indeed, if $\bigcap \pi_i(\mathcal B_j)=\emptyset$, then there exist two boxes $B,B'\in \mathcal B_j$ whose $i$-projections $\pi_i(B)$ and $\pi_i(B')$ are disjoint segments. (It follows for example, from Helly's theorem for line segments on a real line.) Since any two boxes from different families intersect, the $i$-projection of any box $B''$ from another family $\mathcal B_{j'}$ with $j'\neq j$ must intersect both $\pi_i(B)$ and $\pi_i(B')$. Consequently, $\pi_i(B'')$ contains the open segment between the disjoint segments $\pi_i(B)$ and $\pi_i(B')$. Therefore, for every $j'\in J\setminus\{j\}$, the intersection $\bigcap \pi_i(\mathcal B_{j'})$ also contains this open segment and is thus nonempty.

\smallskip

Our goal in this section is to prove Lemma~\ref{claim induction reduction}, which establishes an upper bound for $\tau_n'(\mathcal B_j)$ for some $j\in J$. The obtained bound depends on some auxiliary notation, which we introduce below, along with less formal remarks explaining the meaning of the introduced objects.

The proof of Lemma~\ref{claim induction reduction} relies on the $n$-intersecting property of the trace of any transversal, though we must carefully choose the transversals we work with. Our choice is determined by a sequence $m$ consisting of all but one of the coordinate axes of \(\mathbb{R}^d\), together with a direction (left or right) for each.

Formally, let \(\mathcal{M}\) be the set of all sequences \(\big((i_s,\varepsilon_s)\big)_{s=1}^{d-1}\) such that \(i_1,\dots, i_{d-1}\) are distinct elements of \(I\), and \(\varepsilon_1,\dots, \varepsilon_{d-1} \in \{L, R\}\).

For each $m=\big((i_s,\varepsilon_s)\big)_{s=1}^{d-1}\in \mathcal M$, we introduce the following notation:
\begin{itemize}
    \item[$(i)$]  Let \(i(m)\) be the unique element of the singleton \(I \setminus \{i_1, \dots, i_{d-1}\}\), that is, $i(m)$ is the index of the only coordinate axis that does not appear in $m$.
\end{itemize}

Next, given $m\in \mathcal M$, we determine the transversals used in the proof of Lemma~\ref{claim induction reduction}. To this end, we pick $2d-2$ boxes from different families and consider all transversals sharing the chosen boxes. We pick those boxes in two stages. First, we select $d-1$ boxes, denoted below by $B_{u_s}$ for $s\in [d-1]$, such that $B_{u_s}$ is extremal, in a certain sense, with respect to the $i_s$-th coordinate axis and the direction $\varepsilon_s$.

\begin{itemize}
    \item[$(u)$] Let us define a sequence \(u_1, \dots, u_{d-1} \in J\) of distinct indices inductively. In addition to this, we also determine the sequences of boxes \(B_{u_1} \in \mathcal{B}_{u_1}, \dots, B_{u_{d-1}} \in \mathcal{B}_{u_{d-1}}\).

 Suppose that for some $s\in[d-1]$, the indices $u_1,\dots, u_{s-1}$ have been determined. Let $u_s$ be a unique $t\in J\setminus \{u_1,\dots, u_{s-1}\}$  such that
    \begin{itemize}
        \item[$(L)$] If \(\varepsilon_s = L\), then the family \(\mathcal{B}_t\) contains the box whose \(i_s\)-projection has the largest \textit{left} endpoint among all the boxes in the families \(\mathcal{B}_j\) with \(j \in J \setminus \{u_1, \dots, u_{s-1}\}\). Let \(B_{u_s}\) be the corresponding box in \(\mathcal{B}_{u_s}\). 

        \item[$(R)$] If \(\varepsilon_s=R\), then the family \(\mathcal{B}_t\) contains the box whose \(i_s\)-projection has the smallest \textit{right} endpoint among all the boxes in the families \(\mathcal{B}_j\) with \(j \in J \setminus \{u_1, \dots, u_{s-1}\}\). Let \(B_{u_s}\) be the corresponding box in \(\mathcal{B}_{u_s}\).
    \end{itemize}
\end{itemize}
    \begin{remark} 
    \label{remark definition of u and Helly}
    If for some $s\in[d-1]$ and $j\in J$, we have $\bigcap \pi_{i_s}(\mathcal B_{j})= \emptyset$, then $j\in \{u_1,\dots, u_s\}$. Indeed, if $j\not\in \{u_1,\dots, u_{s-1}\}$, then among $j'\in J\setminus \{u_1,\dots, u_{s-1}\}$, the index $j$ is the only index with $\bigcap \pi_{i_s}(\mathcal B_{j'})= \emptyset$ by~\eqref{equation colorful Helly for R1}. Therefore, the boxes with the largest left and smallest right endpoints among all the boxes in the families $\mathcal B_{j'}$ with $j'\in J\setminus\{u_1,\dots, u_{s-1}\}$ must belong to $\mathcal B_{j}$.
    \end{remark}

\begin{figure}[!h]
    \centering
    
    \includegraphics[scale=0.5]{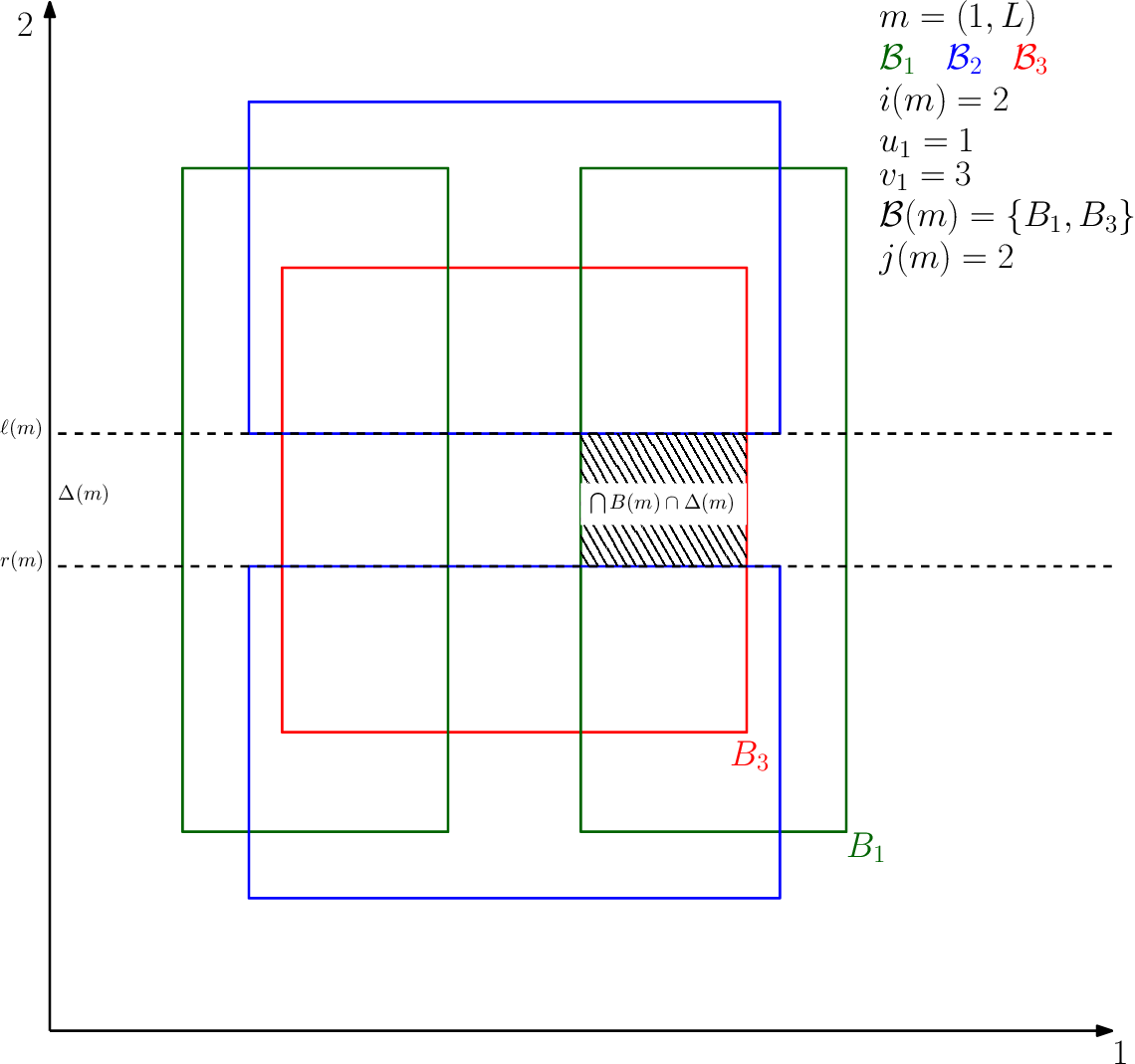}
    
    \caption{Illustration of definitions}
    \label{fi1}
\end{figure}

Next, we select the remaining $d-1$ boxes, denoted below by $B_{v_s}$ for $s\in [d-1]$, from other families (also different) such that each $B_{v_s}$ is extremal, in a certain sense, with respect to the $i_s$-th coordinate axis and the direction opposite to $\varepsilon_s$.

\begin{itemize}
\item[$(v)$] Let us define a sequence \(v_1, \dots, v_{d-1} \in J\setminus \{u_1,\dots, u_{d-1}\}\) of distinct indices inductively. In addition to this, we also determine the sequences of boxes \(B_{v_1} \in \mathcal{B}_{v_1}, \dots, B_{v_{d-1}} \in \mathcal{B}_{v_{d-1}}\).

 Suppose that for some $s\in[d-1]$, the indices $v_1,\dots, v_{s-1}$ have been determined. Let $v_s$ be the unique $t\in J\setminus \{u_1,\dots, u_{d-1},v_1,\dots, v_{s-1}\}$  such that
    \begin{itemize}
        \item[$(L)$] If \(\varepsilon_s = L\), then the family \(\mathcal{B}_t\) contains the box whose \(i_s\)-projection has the smallest \textit{right} endpoint among all the boxes in the families \(\mathcal{B}_j\) with \(j \in J \setminus \{u_1,\dots, u_{d-1},v_1, \dots, v_{s-1}\}\). Let \(B_{v_s}\) be the corresponding box in \(\mathcal{B}_{v_s}\).

        \item[$(R)$] If \(\varepsilon_s=R\), then the family \(\mathcal{B}_t\) contains the box whose \(i_s\)-projection has the largest \textit{left} endpoint among all the boxes in the families \(\mathcal{B}_j\) with \(j \in J \setminus \{u_1,\dots, u_{d-1}, v_1, \dots, v_{s-1}\}\). Let \(B_{v_s}\) be the corresponding box in \(\mathcal{B}_{v_s}\).
    \end{itemize}

\item[$(\mathcal B)$] Let $\mathcal B(m)=\{B_{u_1},\dots,B_{u_{d-1}},B_{v_1},\dots, B_{v_{d-1}}\}$. In the proof of Lemma~\ref{claim induction reduction}, we consider all the transversals sharing $\mathcal B(m)$.

\item[$(j)$] Let $j(m)$ be the unique element of the singleton $J\setminus \{u_1,\dots, u_{d-1},v_1,\dots, v_{d-1}\}$. This is the index of the family for which we are able to prove a reasonable upper bound for $\tau_n'$ in Lemma~\ref{claim induction reduction}.

\end{itemize}

    Next, we determine a region in $\mathbb R^d$, denoted below by $\Delta(m)$, needed to bound $\tau_n'(\mathcal B_{j(m)}|_P)$; see the statement of Lemma~\ref{claim induction reduction}.

\begin{itemize}
\item[$(\Delta)$] Let $\ell(m)$ and $r(m)$ be the largest left and the smallest right endpoints, respectively, of the segments in the family $\pi_{i(m)}(\mathcal B_{j(m)})$. Note that $\bigcap \pi_{i(m)}(\mathcal B_{j(m)}) = \emptyset$ if and only if $\ell(m) > r(m)$. If $\bigcap \pi_{i(m)}(\mathcal B_{j(m)}) \ne \emptyset$, let $\Delta(m)=\emptyset$. Otherwise, $\Delta(m)$ be the set of points in $\mathbb{R}^d$ whose $i(m)$-projections belong to the open line segment $(r(m),\ell(m))$. 

\end{itemize}

\begin{lemma}
\label{claim induction reduction}
    For any $m\in \mathcal M$, we have
    \[
        \tau_n'(\mathcal B_{j(m)}|_P)\leq 2n+ | \bigcap \mathcal B(m)\cap \Delta(m)\cap P|.
    \]
\end{lemma}
\begin{remark*} 
    In particular, by the definition of $\Delta(m)$, if $\pi_{i(m)}(\bigcap \mathcal B_{j(m)})\ne \emptyset$, then $\tau_n'(\mathcal B_{j(m)}|_P)\leq 2n$.
\end{remark*}
\begin{proof}

The proof of the lemma naturally reduces to showing the following two inequalities
\begin{equation}
\label{equation reduction to one dimensional question}
    \tau_n'(\mathcal B_{j(m)}|_P)\leq \tau_n'(\pi_{i(m)}(\mathcal B_{j(m)})|_{\pi_{i(m)}(\bigcap \mathcal B(m)\cap P)})\leq 2n + \big| \bigcap \mathcal B(m)\cap \Delta(m)\cap P\big|.
\end{equation}

\textit{Proof of the first inequality of~\eqref{equation reduction to one dimensional question}.} Let $m=\big((i_s, \varepsilon_s)\big)_{s=1}^{d-1}\in \mathcal M$. Essentially, the proof relies on~\eqref{equation other projections}, the reduction property of $i(m)$-projections.

For each $s\in[d-1]$ and $B\in \mathcal B_{j(m)}$, we have
\begin{equation*}
    \pi_{i_s}(B_{u_s}\cap B_{v_s})\subseteq \pi_{i_s}(B_{u_s})\cap \pi_{i_s}(B_{v_s})\subset \pi_{i_s}(B).
\end{equation*}
Indeed, if $\varepsilon_s=L$, then the left endpoint of the segment $\pi_{i_s}(B_{u_s})$ is greater than the left endpoint of $\pi_{i_s}(B)$ because $B\in \mathcal B_{j(m)}$ and $j(m)$ distinct from $u_1,\dots, u_{s-1}$. Analogously, the right endpoint of $\pi_{i_s}(B_{v_s})$ is smaller than the right endpoint of  $\pi_{i_s}(B)$. Combining this two observation, we complete the proof of the inclusion if $\varepsilon_s=L$. Similarly, we address the case $\varepsilon_s=R$.
 
For any $Q\subseteq \bigcap \mathcal B(m)$ and any $B\in \mathcal B_{j(m)}$, we have that
\begin{equation}
\label{equation other projections}
\text{if }\pi_{i(m)}(Q)\subseteq \pi_{i(m)}(B), \text{\quad then } Q\subset B.
\end{equation}
Indeed, for any $s\in [d-1]$, we obtain that 
\[
\pi_{i_s}(Q)\subseteq \pi_{i_s}\big(\bigcap \mathcal B(m)\big)\subseteq \pi_{i_s} (B_{u_s}\cap B_{v_s})\subset \pi_{i_s}(B),
\]
which together with $\pi_{i(m)}(Q)\subseteq \pi_{i(m)}(B)$ implies that $Q\subset B$.

Now we are ready to complete the argument. Put $x:=\tau_n'(\pi_{i(m)}(\mathcal B_{j(m)})|_{\pi_{i(m)}(\bigcap \mathcal B(m)\cap P)})$. Hence there is $R\subseteq \bigcap  \mathcal B(m)\cap P$ of size $x$ such that there for each $B\in \mathcal B_{j(m)}$, there is a subset $Q_B\subseteq R$ of size $n$ with $\pi_{i(m)}(Q_B)\subseteq \pi_{i(m)}(B)$.  By~\eqref{equation other projections}, for each $B\in \mathcal B_{j(m)}$, we have $Q_B\subset B$, which implies that
\[
    x\geq \tau_n'(\mathcal B_{j(m)}|_{\bigcap \mathcal B(m)\cap P}) \geq \tau_n'(\mathcal B_{j(m)}|_{P}).
\]
This completes the proof of the first inequality of~\eqref{equation reduction to one dimensional question}.
\smallskip

\textit{Proof of the second inequality of~\eqref{equation reduction to one dimensional question}.} By the hypotheses of Theorem~\ref{quantitative colorful Halman 2d-1 colors}, for any \( B \in \mathcal{B}_{j(m)} \), the intersection \( \bigcap \mathcal{B}(m) \cap B \) contains at least \( n \) points of \( P \). Hence, \( \pi_{i(m)}(\mathcal{B}_{j(m)}) \) is a family of segments on the \( i(m) \)th coordinate axis containing at least \( n \) points of \( \pi_{i(m)}(\bigcap \mathcal{B}(m) \cap P) \). In other words, we reduce the problem to 1-dimensional problem about a finite point set of $\mathbb R$ and a family of line segments on $\mathbb R$ such that each segment contains at least $n$ points from the point set.

It remains to find a subset \( S \subseteq \pi_{i(m)}(\bigcap \mathcal{B}(m) \cap P) \) of size at most \( 2n + |\bigcap \mathcal{B}(m)\cap \Delta(m) \cap P| \) such that any segment in \( \pi_{i(m)}(\mathcal{B}_{j(m)}) \) contains at least \( n \) points of $S$. The set \( S \) is constructed by:
\begin{itemize}
    \item including the \( n \) largest points of \( \pi_{i(m)}(\bigcap\mathcal{B}(m) \cap P) \) that are less than \( r(m) \);
    \item including the \( n \) smallest points of \( \pi_{i(m)}(\bigcap\mathcal{B}(m) \cap P) \) that are greater than \( \ell(m) \);
    \item if \( \ell(m) > r(m) \), also including all points of \( \pi_{i(m)}(\bigcap\mathcal{B}(m) \cap P) \) lying in the open interval \( (r(m), \ell(m)) \). Note that the number of such points exactly $|\bigcap \mathcal{B}(m)\cap \Delta(m) \cap P|$.
\end{itemize}

Since any line segment in \( \pi_{i(m)}(\mathcal{B}_{j(m)}) \) has at least \( n \) points from \( \pi_{i(m)}(\bigcap \mathcal{B}(m) \cap P) \), and its left endpoint is no greater than \( \ell(m) \) while the right endpoint is no smaller than \( r(m) \), it is easy to verify that this segment contains at least \( n \) points of \( S\subseteq \pi_{i(m)}(\bigcap \mathcal{B}(m) \cap P) \). This completes the proof of the lemma. 
\end{proof}

    Using essentially the same argument, we can prove the following lemma, which allows to use induction on the dimension in similar colorful problems. We use it in the proof of Corollary \ref{corollary of pq Halman theorem for q at least 2}.
\begin{lemma}
\label{lemma induction reduction baby version}
        Let \( n \geq 1, d\geq 2, s\geq 3,\) and $t\geq 1$ be integers, with $d-t\geq 1$ and $s-2t\geq 1$. Let $P\subset \mathbb R^d$ be a finite set, and let \( \mathcal{B}_1, \dots, \mathcal{B}_{s} \) be finite families of boxes in $\mathbb R^d$. Assume that for any their transversal $\mathcal B$, its trace $\mathcal B|_P$ is \( n \)-intersecting. Then there exist a finite set $P'\subset \mathbb R^{d-t}$ and finite families $\mathcal B'_{j_1},\dots, \mathcal B'_{j_{s-2t}}$ of boxes in $\mathbb R^{d-t}$, where $j_1,\dots, j_{s-2t}\in[k]$ are distinct, such that 
        \begin{itemize}
            \item for any their transversal $\mathcal B'$, its trace $\mathcal B'|_{P'}$ is $n$-intersecting;
            \item $\tau_n'(\mathcal B_{j_i} |_P)\leq \tau_n'(\mathcal B_{j_i}'|_{P'})$ for each $i\in [s-2t]$;
            \item if $\mathcal B_{1}=\dots =\mathcal B_{s}$, then $\mathcal B_{j_1}'=\dots=\mathcal B'_{j_{s-2t}}$.
        \end{itemize}  
    \end{lemma}

\subsubsection{Reduction to Lemma~\ref{lemma on a measure}}

Our first step is to reduce the problem to the case when \eqref{equation assumption} holds. Under this assumption, we state Lemma~\ref{lemma on a measure} and explain how it leads to the proof of the theorem. The proof of the lemma is given in the next section.

Without loss of generality, assume that $\bigcap \pi_{i(m)}(\mathcal B_{j(m)}) =  \emptyset$ for any $m\in \mathcal M$. Indeed, otherwise, $\tau_n'(\mathcal B_{i(m)}|_P)\leq 2n$, which completes the proof of the theorem.  In particular, for each \( i \in I \), there exists \( j \in J \) with \( \bigcap  \pi_i(\mathcal{B}_j) = \emptyset \), which is unique by~\eqref{equation colorful Helly for R1}.

Suppose that there are distinct \( i, i' \in I \) and \( j \in J \) with \( \bigcap \pi_i(\mathcal{B}_j) = \bigcap  \pi_{i'}(\mathcal{B}_j) = \emptyset \). Choose any \( m \in \mathcal M \) with \( i(m) = i'\). Since the intersections $\bigcap \pi_{i(m)}(\mathcal{B}_{j(m)})=\bigcap \pi_{i'}(\mathcal{B}_{j(m)})$ and $\bigcap \pi_{i'}(\mathcal{B}_{j})$ are empty, \eqref{equation colorful Helly for R1} implies $j(m)=j$. However, by Remark~\ref{remark definition of u and Helly}, we get $j\in \{u_1,\dots, u_{d-1}\}$, that is, $j\ne j(m)$, a contradiction. Hence, we can assume that for each \( i \in I \), there exists a unique \( j \in J \) with \( \bigcap \pi_i(\mathcal{B}_j) = \emptyset \), and these \( j \)'s are distinct for distinct \( i \)'s.

Using the above assumptions and recalling that $I=[d]$ and $J=[2d-1]$, we can arrange the index sets $I$ and $J$ in such a way that 
\begin{equation}
\label{equation assumption}
\text{for any $i\in I, j\in J$,}\quad \bigcap \pi_i(\mathcal{B}_j) = \emptyset \text{\quad if and only if } j=i.
\end{equation}

In particular, it means that for any $m=((i_s,\varepsilon_s))_{s=1}^{d-1}\in \mathcal M$, we have that the indices $u_1,\dots, u_{d-1}$, $v_1,\dots,v_{d-1}\in J$ defined in Section~\ref{section notation} satisfy
\begin{equation}
\label{equation: sets I and J without I}
    u_s=i_s\text{ for all } s\in [d-1]\text{\quad and \quad} j(m)=i(m)\in I,
    \text{\quad hence, \quad}\{v_1,\dots, v_{d-1}\}=J\setminus I.
\end{equation}

To show that, we first note that by~\eqref{equation assumption}, for each $i\in I$, among all the boxes in the families $\mathcal B_j$, $j\in J$, the box with the largest left endpoint of its $i$-projection must belong to $\mathcal B_i$. (In this paragraph, we argue that this is true; note that exactly the same argument works if we replace `the largest left endpoint' by `the smallest right endpoint'.) Indeed, since $\bigcap \pi_i(\mathcal B_i)=\emptyset$, there are two boxes $B,B'\in \mathcal B_i$ with disjoint $i$-projections $\pi_i(B)$ and $\pi_i(B')$. (For example, it follows from the Helly theorem for $\mathbb R^1$.) Since any box $B''$ of the families $\mathcal B_j$, $j\in J\setminus \{i\}$, must intersect $B$ and $B'$, its $i$-projection must contain the open line segment between the disjoint segments $\pi_i(B)$ and $\pi_i(B')$. In particular, this implies that among the $i$-projections $\pi_i(B),\pi_i(B'),\pi_i(B'')$, the segment $\pi_i(B'')$ does not have the largest left endpoint, and so, the box with the largest left endpoint of its $i$-projection must belong to $\mathcal B_i$.

Given this observation, we can prove~\eqref{equation: sets I and J without I}. First, we show by induction that $i_s=u_s$ for any $s\in [d-1]$. Assume that it holds for all $s< s_0\in [d-1]$ and consider the case $s=s_0$. By the observation from the last paragraph, we have that the family $\mathcal B_{i_{s_0}}$ contains the box whose $i_s$-projection has the largest left endpoint (as well as the box whose $i_s$-projection has the smallest right endpoint) among all the boxes in the families $\mathcal B_j$, $j\in J\setminus \{i_1,\dots, i_{s_0}\}$. Hence, $u_{s_0}=i_{s_0}$.

To complete the proof of~\eqref{equation: sets I and J without I}, it remains to show $j(m)=i(m)$. Indeed, by the assumption given at the beginning of this section, we have that $\bigcap \pi_{i(m)}(\mathcal B_{j(m)})=\emptyset$. Hence,~\eqref{equation assumption} yields $j(m)=i(m)$.
\smallskip 

Finally, we are ready to state Lemma~\ref{lemma on a measure} under the assumption~\eqref{equation assumption}.
\begin{lemma}
\label{lemma on a measure}
    There exists $\mathcal M' \subseteq  \mathcal M$ of size $d(d-1)\cdot 2^{d-1}$ such that for any $p\in\mathbb R^d$, we have
    \[
        \Big| \big\{m \in  \mathcal M'  \mid  p \in \bigcap \mathcal B(m)\cap \Delta(m)   \big\} \Big|
        \leq
        \Big|\big\{j\in J\setminus I\mid p\in \bigcap \mathcal B_j\big\}\Big|.
    \]
\end{lemma}

If $\mathcal M'$ is the subset from Lemma~\ref{lemma on a measure}, then we have
\begin{align*}
    \sum_{m \in \mathcal M'}\big|\bigcap \mathcal B(m)\cap \Delta(m) \cap P\big| 
     &\quad = \sum_{p\in P} \Big| \big\{m \in  \mathcal M'  \mid  p \in \bigcap \mathcal B(m) \cap \Delta(m) \big\}  \Big| &\text{(by double counting)}\\
    &\quad\leq\quad \sum_{p\in P} \Big|\big\{j\in J\setminus I\mid p\in \bigcap \mathcal B_j\big\}\Big| &\text{(by Lemma~\ref{lemma on a measure})}\\
    &\quad = \quad  \sum_{j\in J\setminus I}\Big| \bigcap \mathcal B_j \cap P \Big|. &\text{(by double counting)}
\end{align*} 

For each $j\in J\setminus I$, without loss of generality, we can additionally assume that 
\begin{equation*}
    \big|\bigcap \mathcal{B}_j \cap P\big| \leq n - 1.
\end{equation*}
Indeed, otherwise, for some $j\in J\setminus I$, the intersection $\bigcap \mathcal B_j$ contains at least $n$ points of $P$, and so, $\tau_n'(\mathcal B_j)=n$, which completes the proof of the theorem.

Hence, by $|J\setminus I|=d-1$, we conclude
\[
\sum_{m \in \mathcal M'}\big|\bigcap \mathcal B(m)\cap \Delta(m) \cap P\big| \leq  (d-1)(n-1).
\]
Since $|\mathcal M'|=d(d-1)\cdot 2^{d-1}$, there exist $m\in \mathcal M'$ with
\[
    \big|\mathcal B(m) \cap \Delta(m)\cap P \big| \leq \frac{n-1}{d\cdot 2^{d-1}},
\]
which combined with Claim~\ref{claim induction reduction} completes the proof of Theorem~\ref{quantitative colorful Halman 2d-1 colors}.

\subsubsection{Proof of Lemma~\ref{lemma on a measure}}
\label{subsection proof of the key lemma}

Our proof breaks naturally in two steps. To present them, first we clarify that by the underlying set of $((i_s,\varepsilon_s))_{s=1}^{d-1}\in \mathcal M$, we mean the set $\{(i_s,\varepsilon_s)\}_{s=1}^{d-1}$.

For each possible underlying set $T$ of sequences of $\mathcal M$, we find a subset $\mathcal M_T \subseteq \mathcal M$ consisting of $d-1$ sequences with the same underlying set $T$, such that for every $p \in \mathbb R^d$ we have
\begin{equation}
\label{equation final inequality for T}
\Big| \big\{m \in \mathcal M_T \mid p \in \bigcap \mathcal B(m)\cap \Delta(m)\big\} \Big|
\leq
\Big|\big\{j\in J\setminus I \mid p\in \bigcap \mathcal B_j\big\}\Big|.
\end{equation}

Apart from that, we show that for $m,m'\in \mathcal M$ with distinct underlying sets, 
\begin{equation}
\label{equation things are disjoint}
\big(\bigcap\mathcal B(m)\cap \Delta(m)\big) 
\cap 
\big(\bigcap\mathcal B(m')\cap \Delta(m')\big)=\emptyset.
\end{equation}

Given~\eqref{equation things are disjoint} and~\eqref{equation final inequality for T}, it easy to complete the proof of the lemma. Indeed, since there are $d\cdot 2^{d-1}$ possible underlying sets $T$ of sequences in $\mathcal M$, the set $\mathcal M':=\bigsqcup_{T}\mathcal M_T$ contains exactly $d(d-1)\cdot 2^{d-1}$ distinct sequences of $\mathcal M$. By~\eqref{equation things are disjoint}, among the sets $\{m \in  \mathcal M_T  \mid  p \in \bigcap \mathcal B(m)\cap \Delta(m)\}$ for different underlying $T$ and a fixed $p\in \mathbb R^d$, there is at most one nonempty. Combining this with~\eqref{equation final inequality for T}, we conclude the desired inequality.\smallskip

We start with proving~\eqref{equation things are disjoint}. For each $i\in I$ and $j\in J$, let $\ell_{i,j}$ and $r_{i,j}$ be the largest left and the smallest right endpoints, respectively, of the segments in the family $\pi_i(\mathcal B_j)$. By~\eqref{equation assumption}, for any $i\in I$ and $j\in J\setminus \{i\}$, we have 
\[
\ell_{i,j}<r_{i,i}<\ell_{i,i}<r_{i,j}.
\]
Using the notation from Section~\ref{section notation}, we have $r(m)=r_{i(m),j(m)}$ and $\ell(m)=\ell_{i(m),j(m)}$.

Let $m=\big((i_s,\varepsilon_s)\big)_{s=1}^{d-1}\in \mathcal M$. For each $i\in I$, the projection $\pi_i(\bigcap \mathcal B(m)\cap \Delta(m))$ belongs to one of three disjoint regions: the closed rays
\((-\infty, r_{i,i}] \) and \( [\ell_{i,i}, +\infty),\) and the open segment $(r_{i,i},\ell_{i,i})$. Indeed, if $i=i(m)$, then by~\eqref{equation: sets I and J without I} and the definition of $\Delta(m)$ in Section~\ref{section notation}, we have $i=i(m)=j(m)$ and
\begin{equation*}
    \pi_{i}(\Delta(m))=(r(m),\ell(m))=(r_{i,i}, \ell_{i,i}).
\end{equation*}
If there is $s\in [d-1]$ such that $i=i_s$, by~\eqref{equation: sets I and J without I}, we have $i_s=u_s=i$, and so,
\begin{equation*}
\label{equation cap B(m) projection}
    \pi_{i}(\bigcap \mathcal B(m)) \subset \pi_{i}(B_{u_s})=\pi_{i}(B_{i_s})\subseteq
    \begin{cases}
        [\ell_{i,i}, +\infty) \text{\quad if }\varepsilon_s=L\\
         (-\infty, r_{i,i}] \text{\quad if }\varepsilon_s=R.
 \end{cases}
\end{equation*}

Combining the above observations, we easily conclude~\eqref{equation things are disjoint}. Indeed, since $m, m' \in \mathcal M$ have distinct underlying sets, there are two possible cases. The first one is when $i(m) \neq i(m')$. In this case, note that $\pi_{i(m)} (\bigcap \mathcal B(m)\cap \Delta(m))$ is a subset of $ (r_{i(m),i(m)} , \ell_{i(m),i(m)})$ while $\pi_{i(m)} (\bigcap \mathcal B(m')\cap \Delta(m'))$ is a subset of one of the rays $(-\infty, r_{i,i}]$ or $[\ell_{i,i},+\infty)$. The second case is when there is $i\in I$ such that one of $(i,L)$ and $(i,R)$ belongs to $m$ and the other belongs to $m'$. Then note that $\pi_{i} (\bigcap \mathcal B(m)\cap \Delta(m))$ and $\pi_{i} (\bigcap \mathcal B(m')\cap \Delta(m'))$ are subsets of the disjoint rays $(-\infty, r_{i,i}]$ and $[\ell_{i,i},+\infty)$.\smallskip

We are left to prove~\eqref{equation final inequality for T} for any fixed $p\in \mathbb R^d$.  Without loss of generality, it is enough to find a desired set $\mathcal M_T$ of size $d-1$ for one specific $T$. We assume that $T=\big\{(i,L)\mid i\in I\setminus \{d\}\big\}.$ Moreover, we can assume that a fixed point $p\in \mathbb R^d$ satisfies the inequalities
\begin{equation}
    \label{equation inequalities for p}
    \pi_i(p)\geq \ell_{i,i}\text{\quad for any }i\in I\setminus \{d\}\text{\quad and } r_{d,d}<\pi_{d}(p)<\ell_{d,d}.
\end{equation}
Indeed, otherwise, according to the argument for~\eqref{equation things are disjoint} the set $\{m\in \mathcal M_T \mid p \in \bigcap \mathcal B(m)\cap \Delta(m)\big\}$ is empty, and so~\eqref{equation final inequality for T} holds.

From now on, we identify the set $[d-1]=I\setminus \{d\}$ with the cyclic group $\mathbb Z_{d-1}$. Consider the following $d-1$ sequences  
\[
m_i:=\big((s+i-1, L)\big)_{s=1}^{d-1}\in \mathcal M\text{\quad for }i\in \mathbb Z_{d-1}
\]
with the same underlying set and $i(m_i)=d$. To complete the proof, it is enough to verify the inequality
    \[
        \Big| \big\{i\in \mathbb Z_{d-1}  \mid  p \in \bigcap \mathcal B(m_i)\cap \Delta(m_i)   \big\} \Big|
        \leq
        \Big|\big\{j\in J\setminus I\mid p\in \bigcap \mathcal B_j\big\}\Big|.
    \]
for any $p$ satisfying~\eqref{equation inequalities for p}.

Let us choose any $i\in \mathbb Z_{d-1}$ and consider the sequence $v_1,\dots, v_{d-1}\in J\setminus I$ for $m_i\in \mathcal M_T$; see the definition in Section~\ref{section notation}.  Next, we prove the crucial property.

We claim that for any $j\in J\setminus I, s\in [d-1]$,
\[
\text{if }p\in \bigcap \mathcal B(m_i)\text{ and }    \pi_{i+s-1}(p) > r_{i+s-1,j} \text{,\quad then }j\in \{v_1,\dots, v_{s-1}\}.
\]
Indeed, by the definition of $v_s$ (for the sequence $m_i$), the point $r_{i+s-1,v_s}$ is the smallest element among $r_{i+s-1,h}$ for $h\in J\setminus (I\cup \{v_1,\dots, v_{s-1}\})$. Since $p\in \bigcap \mathcal B(m_i) \subseteq B_{v_s}$, we have $r_{i+s-1,v_s}\geq \pi_{i+s-1}(p)>r_{i+s-1,j}$, which means that $j\not \in  J\setminus (I\cup \{v_1,\dots, v_{s-1}\})$. Therefore, $j\in \{v_1,\dots, v_{s-1}\}$.

By letting
\[
    W_s:=\big\{j\in J\setminus I\mid \pi_s(p)>r_{s,j}\big\},
\]
we conclude that for any $s\in [d-1]$, 
\[
    \text{if }p\in \bigcap \mathcal B(m_i),\text{\quad then } W_i\cup \dots\cup W_{i+s-1}\subseteq \{v_1,\dots, v_{s-1}\}.
\]
Hence 
\[
|W_i\cup \dots \cup W_{i+s-1}|\leq s-1.
\]

Suppose the set $\{i\in \mathbb Z_{d-1}\mid p\in \bigcap \mathcal B(m_i)\}$ has size $k$ and let $1\leq i_1<\dots<i_k\leq d-1$ integers corresponding to its elements. Next, we apply the above inequality for $i=i_t$ and $s=i_{t+1}-i_t$ if $1\leq t\leq k-1$ and for $i=i_k$ and $s=i_1+d-1-i_{k}$. Summing up these $k$ inequalities, we obtain
\begin{align*}
    |W_1\cup\dots \cup W_{d-1}|&\leq \sum_{t=1}^{k-1} |W_{i_t}\cup\dots W_{i_{t+1}-1}|+|W_{i_k}\cup\dots\cup W_{d-1}\cup W_1\cup \dots \cup W_{i_1+d-2}|\\
    &\leq \sum_{t=1}^{k-1} (i_{t+1}-i_t-1)+(i_1+d-1-i_k-1)\\
    &=d-1-k\\
    &=d-1-\big|\{i\in\mathbb Z_{d-1}\mid p\in \bigcap \mathcal B(m_i)\}\big|,
\end{align*}
which implies that
\begin{align*}
\big|\{i\in\mathbb Z_{d-1}\mid p\in \bigcap \mathcal B(m_i)\cap \Delta(m_i)\}\big|&=\big|\{i\in\mathbb Z_{d-1}\mid p\in \bigcap \mathcal B(m_i)\}\big|&\text{(by \eqref{equation inequalities for p})}
\\&\leq d-1-\big|\big\{j\in J\setminus I\mid\pi_s(p)> r_{s,j}\text{ for some }s \in [d-1] \big\}\big|
\\&=d-1-\big|\{j\in J\setminus I \mid p\not\in \bigcap \mathcal B_j \big\}\big|&\text{(by \eqref{equation inequalities for p})}\\
&=\big|\{j\in J\setminus I\mid p\in \bigcap \mathcal B_j \big\}\big|.
\end{align*}
This completes the proof of the lemma.

\subsubsection{Proof of the lower bound}\label{subsection: construction of lower bound}
Let us recall the statement about the lower bound. 

\begin{claim}
\label{claim: example}
For any \( n \geq 1 \) and \( d \geq 2 \), there exists a finite set \( P \subset \mathbb{R}^d \) and finite families \( \mathcal{B}_1, \dots, \mathcal{B}_{2d-1} \) of boxes in \( \mathbb{R}^d \) such that for any their transversal \( \mathcal{B} \), its trace \( \mathcal{B}|_P \) is \( n \)-intersecting and satisfies \( \tau_n'(\mathcal{B}_j|_P) = 2n \) for any \( j \in [2d-1] \).
\end{claim}

\begin{proof}
The proof is organized as follows. We outline the structure of the example and state a property ensuring the claim. We then describe the example in detail, verify the property, and complete the proof.\smallskip

For each $i \in [d]$, the family $\mathcal B_i$ consists of two large boxes separated by the hyperplane 
\(\{(x_1,\dots, x_d) \in \mathbb R^d : x_i = 0\}\).
For simplicity, the reader may assume that one of the boxes is the open half-space defined by $x_i > 0$ and the other is the open half-space defined by $x_i < 0$. Note that for each transversal $\mathcal B'$ for $\mathcal B_1,\dots,\mathcal B_d$, the intersection $\bigcap \mathcal B'$  is a box lying in the corresponding open orthant determined by the coordinate hyperplanes.
 
For each $i \in [2d-1] \setminus [d]$, the family $\mathcal B_i$ consists of two boxes such that for each transversal $\mathcal B'$ for $\mathcal B_1,\dots, \mathcal B_d$ and for each transversal $\mathcal B''$ for $\mathcal B_{d+1},\dots, \mathcal B_{2d-1}$, the intersection $\bigcap \mathcal B'$ contains exactly one vertex of $\bigcap \mathcal B''$. In particular, each box from $\mathcal B_{d+1}\cup \dots\cup \mathcal B_{2d-1}$ contains the origin in its interior. Also, we assume that the boxes in the families $\mathcal B_{d+1},\dots, \mathcal B_{2d-1}$ are distinct.

We define $P$ as the multiset consisting of $n$ copies of the set \[
\bigcup_{\mathcal B''} \mathrm{vert}(\bigcap \mathcal B''),
\]
where the union is taken over all transversals $\mathcal B''$ for $\mathcal B_{d+1},\dots,\mathcal B_{2d-1}$, and $\mathrm{vert}(B)$ denotes the vertex set of a box $B\subset\mathbb R^d$.
\smallskip

Our goal is to find an example of families $\mathcal B_1,\dots, \mathcal B_{2d-1}$ satisfying the above description such that \textit{for any transversal $\mathcal B''$ for $\mathcal B_{d+1},\dots, \mathcal B_{2d-1}$, each vertex of the box $\bigcap \mathcal B''$ lies outside of each of the boxes in $(\mathcal B_{d+1}\cup \dots \cup \mathcal B_{2d-1}) \setminus \mathcal B''$.} We claim that if this property holds, then the obtained example satisfies the claim.
\smallskip 

Indeed, any transversal $\mathcal B$ for $\mathcal B_1,\dots, \mathcal B_{2d-1}$ is the union of the transversal $\mathcal B'$ for $\mathcal B_1,\dots, \mathcal B_d$ and the transversal $\mathcal B''$ for $\mathcal B_{d+1},\dots, \mathcal B_{2d-1}$. Note that exactly one vertex of $\bigcap \mathcal B''$, which corresponds to the $n$ points of $P$, lies in $\bigcap \mathcal B'$. Hence, $\bigcap \mathcal B=(\bigcap \mathcal B')\cap (\bigcap \mathcal B'')$ contains exactly $n$ points of $P$.

Let us show that $\tau_n'(\mathcal B_i) = 2n$ for all $i \in [2d-1]$. 
For each $i \in [d]$, the two boxes in $\mathcal B_i$ are disjoint, and thus $\tau_n'(\mathcal B_i) \geq 2n$. 
For $i \in [2d-1] \setminus [d]$, by the assumed property and the construction of $P$, the two boxes of $\mathcal B_i$ do not share any point of $P$, and so $\tau_n'(\mathcal B_i) \geq 2n$.
Since each family $\mathcal B_i$ contains exactly 2 boxes, we conclude that $\tau_n'(\mathcal B_i) \leq 2n$, which completes the verification that Claim~\ref{claim: example} holds for families satisfying the property.
\medskip

Next, we provide a concrete example of families satisfying the above property. For each $i \in [d]$, let $\mathcal B_i =  \{ B_{i,-1} ,B_{i,+1}\}$, where
\[
\begin{aligned}
    B_{1,-1} &= [-N, -\delta] \times [-N, +N] \times \dots \times [-N, +N], \\
    B_{1,+1} &= [+\delta, +N] \times  [-N, +N] \times \dots \times [-N, +N], \\
    \vdots \\
    B_{d,-1} &= [-N, +N] \times \dots \times [-N, +N] \times [-N, -\delta] , \\
    B_{d,+1} &=  [-N, +N] \times \dots \times [-N, +N] \times [+\delta, +N],
\end{aligned}
\] with $N\geq 2d$ and $0<\delta<1$. For simplicity, the reader can assume that $N$ is $+\infty$ and $\delta$ is a positive constant very close to $0$.

For families $\mathcal B_{d+1},\dots, \mathcal B_{2d-1}$, we introduce some auxiliary notions. For \( i \in [d-1] \) and \( k \in [d] \), define \( q_1^i(k) \) and \( q_2^i(k) \) as follows. For \( i \in [d-1]  \), let
\[
q_1^i(k) =
\begin{cases}
d, & k \in [d]\setminus \{d-i+1\}; \\
1,   & k = d-i+1,
\end{cases}
\qquad
q_2^i(k) =
\begin{cases}
d-i,   & k \in \{1,\dots,d-i\}; \\
d, & k \in \{d-i+1,\dots, d-1\}.
\end{cases}
\]  

Using the above notions, let use define the boxes of the families $\mathcal{B}_{d+i} = \{B_{d+i,1}, B_{d+i,2}\}$ for each \( i \in [d-1] \) as follows. Put
\[
    B_{d+i,j} = \big[-q_j^i(1),\, +q_j^i(1)\big] \times \dots \times \big[-q_j^i(d),\, +q_j^i(d)\big] \text{ \quad for any }i\in[d-1], j\in \{1,2\}. 
\]  
One can easily verify that for any transversal $\mathcal B'$ for $\mathcal B_1,\dots, \mathcal B_d$ and any box $B$ from $\mathcal B_{d+1},\dots, \mathcal B_{2d-1}$, only one vertex of $B$ belongs to $\bigcap \mathcal B'$, that is, the families are constructed as described at the beginning of the proof.

For each \( \varepsilon = (\varepsilon_1, \dots, \varepsilon_d) \in \{\pm 1\}^d \) and \( s = (s_1, \dots, s_{d-1}) \in \{1,2\}^{d-1} \), the point  
\[
p_s^\varepsilon=\Big( \varepsilon_1 \min_{i \in [d-1]} q_{s_i}^i(1), \, \dots, \, \varepsilon_d \min_{i \in [d-1]} q_{s_i}^i(d) \Big),
\] 
which is a vertex of $B_{d+1,s_1}\cap \dots \cap B_{2d-1,s_{d-1}}$, which lies in $B_{1,\varepsilon_1}\cap \dots \cap B_{d,\varepsilon_d}$.

The property we aim to verify asserts that for any $\varepsilon\in \{\pm 1\}^d$, any $s\in \{1,2\}^{d-1}$, and $i\in [d-1]$, the point $p_s^\varepsilon$ does not belong to the box $B_{d+i, 3-s_i}$. By symmetry, it suffices to assume that $\varepsilon=(+1,\dots, +1)$. Using the notion
\[
    p_s(k)=\min_{i\in[d-1]} q^i_{s_i}(k) \text{\quad for each }s=(s_1,\dots,s_{d-1})\in\{1,2\}^{d-1}, k\in [d],
\]
we can rephrase the property as follows. \textit{For each $s=(s_1,\dots, s_{d-1})\in \{1,2\}^{d-1}$ and $i\in[d-1]$, there is $k\in[d]$ such that 
\[ 
    p_s(k)> q^i_{3-s_i}(k).
\]}

Let us verify this property for any fixed $s\in\{1,2\}^{d-1}$ and $i\in[d-1]$. If $s_{i'}=2$ for some $i' \in [d-1]\setminus [i-1]$, let $i_0$ be the smallest such index. If there is no such $i'$, set $i_0=d$. By choosing $k=d-i_0+1$, we have
\[ 
\begin{aligned}
p_s(k) &\geq \min \big\{
 q_1^{j_1}(k),q_2^{j_2}(k) \mid j_1 \in [d-1]\setminus \{i_0\},  j_2 \in \{1,\dots, i-1,i_0,\dots, d-1\} \big\}  \\
&=\min  \big\{d, \dots, d; d-1,\dots,d-(i-1); d,\dots,d\big\}= d-i+1.
\end{aligned}
 \]
We have two cases. If $s_i=1$, then $q_2^i(d-i_0+1) = d-i < d-i+1 \leq p_s(d-i_0+1)$. If $s_i=2$, that is, $i_0=i$, then $q_1^i(d-i+1) = 1 < d-i+1 \leq p_s(d-i+1)$.

\end{proof}

\subsection{Quantitative Halman theorem for subfamilies of size $2d-1$}\label{section: monochromatic version}
\begin{theorem}\label{monochromatic version}
Let \( n \geq 1 \) and \( d \geq 2 \) be integers. Let $P\subset \mathbb R^d$ be a finite set, and let $\mathcal B$ be a finite family of boxes in $\mathbb R^d$. If for every subfamily $\mathcal B' \subseteq \mathcal B$ of size $2d-1$, its trace $\mathcal B'|_P$ is \( n \)-intersecting, then $\tau_n'(\mathcal B|_P)\leq 2n$.
\end{theorem}

\begin{proof}
By setting $\mathcal B_1 = \dots = \mathcal B_{2d-1}  = \mathcal B$, we may apply Theorem~\ref{quantitative colorful Halman 2d-1 colors}. However, to obtain a tighter upper bound, we need to slightly adjust the proof of this theorem. 

By Lemma~\ref{claim induction reduction}, for any $m\in \mathcal M$, we have 
\[
\tau_n'(\mathcal B|_P) =\tau_n'(\mathcal B_{j(m)}|_P) \leq 2n + \big| \bigcap \mathcal B(m)\cap \Delta(m)\cap P\big|.
\]
Since any two boxes of $\mathcal B$ intersects, we have $\bigcap \mathcal B_{j(m)}\ne \emptyset$ for any $m$. Hence $\bigcap \pi_{i(m)}(\mathcal B_{j(m)})\ne \emptyset$, and by definition of $\Delta(m)$, we conclude that $\Delta(m) = \emptyset$ for any $m \in \mathcal M$. This completes the proof of the theorem.
\end{proof}

The following example shows that for any $n\geq 2$ and $d\geq 1$, in $\mathbb R^d$, there are a finite set $P$ and a family $\mathcal B$ of boxes such that for any $\mathcal B'\subseteq \mathcal B$ of size at most $2d-1$, we have $\mathcal B|_P$ is $n$-intersecting, and $\tau_n'(\mathcal B|_P) \geq  \frac{2d}{2d-1} n $.

\begin{example} 
\label{example lower bound for monochromatic version}
Assume that $\mathcal B=\{B_1,\dots, B_{2d}\}$ is a family of boxes in $\mathbb R^d$ and $P=\{p_1,\dots, p_{2dn}\}$ is a multiset of points in $\mathbb R^d$, with elements defined as follows:
\[
\begin{aligned}
B_1 &= [-1,0] \times [-1,1] \times \dots \times [-1,1], 
&\quad p_1 = \dots = p_n = (1,0,\dots,0); \\
B_2 &= [0,1] \times [-1,1] \times \dots \times [-1,1], 
&\quad p_{n+1} = \dots = p_{2n} = (-1,0,\dots,0); \\
&\vdots \\[6pt]
B_{2d-1} &= [-1,1] \times \dots \times [-1,1] \times [-1,0], 
&\quad p_{(2d-2)n+1} = \dots = p_{(2d-1)n} = (0,\dots,0,1); \\
B_{2d} &= [-1,1] \times \dots \times [-1,1] \times [0,1], 
&\quad p_{(2d-1)n+1} = \dots = p_{2dn} = (0,\dots,0,-1).
\end{aligned}
\]

For any $i \in [2d]$ and $j\in [2d]\setminus \{i\}$, points $p_{(i-1)n+1}, \dots,  p_{in}$ lie in box $B_{j}$. Hence, the intersection of any $2d-1$ boxes of $\mathcal B$ contains exactly $n$ points of $P$. 

Each $n$-piercing multiset of $\mathcal B$ can be represented as the union of multisets $P_1 \cup \dots \cup P_{2d}$, where $P_i \subseteq \{p_{(i-1)n+1}, \dots, p_{in}\}$. Hence, to find $\tau_n'(\mathcal B|_P)$, it is enough to solve the following integer program
\begin{align*}
   & \tau_n'(\mathcal B|_P) = \min \sum_{i \in [2d]} |P_i|, \\
   & \text{subject to } \sum_{j \in [2d] \setminus \{i\}} |P_j| \geq n \text{\quad and } |P_i|\in \{0,1,\dots, n\}\quad \text{for all } i \in [2d].
\end{align*}
A standard argument shows that $\tau_n'(\mathcal B|_P) \geq   \frac{2d}{2d-1} n $.
\end{example}

\section{$(p,q)$-type Halman theorem}\label{section pq Halman theorem}
Our proof of Theorem~\ref{pq Halman theorem for q at least 2} closely follows the proof of Theorem~1.10 in \cite{keller2020p} of Keller and Smorodinsky, with suitable modifications to handle the present setting.

\begin{proof}[Proof of Theorem~\ref{pq Halman theorem for q at least 2}]
Recall the classical Ramsey theorem \cite{ramsey1930problem} and Theorem 1.3 from \cite{eom2025fractionaldiscretehellypairs}.

\begin{theorem}\label{ramsey}
    Given positive integers $s,n$, there exists $r(s,n)$, which is the minimum positive integer, such that any graph on $r(s,n)$ vertices contains a complete graph on $s$ vertices or an independent set on $n$ vertices.
\end{theorem}

\begin{theorem}[Theorem 1.3 in \cite{eom2025fractionaldiscretehellypairs}]\label{leeom}
    There exists a function $M: \mathbb N \to \mathbb N$ such that the following holds. For any finite set $ P \subseteq \mathbb R^d$ and any family $\mathcal B$ of boxes in $\mathbb R^d$ with $|\mathcal B| = M(d)$, if  $\mathcal B|_P$ has $(2,2)$-property, then there exists a subfamily $\mathcal B' \subseteq \mathcal B$ of size $d+1$ such that its trace $\mathcal B'|_P$ is intersecting.
\end{theorem}
    
Without loss of generality, we can assume that the family $\mathcal B|_P$ has $(p,2)$-property, where $p\geq M(d)$, because all other cases follow from this case. If we show that $\mathcal B|_P$ has the $(r(p,p), d+1)$ property, we can apply Theorem~\ref{pq Halman theorem for q at least d plus 1} and complete the proof.

Consider a graph, whose vertices are subsets in $\mathcal B|_P$ and edges are pairs of intersecting subsets. To prove $(r(p,p), d+1)$-property of $\mathcal B$, consider the induced subgraph on any subfamily $\mathcal B'\subseteq \mathcal B$ of size $r(p,p)$. Since $\mathcal B'|_P$ has $(p,2)$-property, there are no independent subsets of size $p$. By the Ramsey theorem, there is a clique of size $p$, that is, $\mathcal B'|_P$ contains $p$ pairwise intersecting sets. Therefore, there is $\mathcal B''\subset \mathcal B'$ of size $p$ such that $\mathcal B''|_P$ has $(2,2)$-property. By Theorem~\ref{leeom}, there is a subfamily $\mathcal B'''\subseteq \mathcal B''$ of size $d+1$ such that its trace $\bigcap \mathcal B'''|_P$ is intersecting, which completes the proof.
\end{proof}
\begin{example}    

\begin{figure}[!h]
    \centering
    
    \includegraphics[scale=0.5]{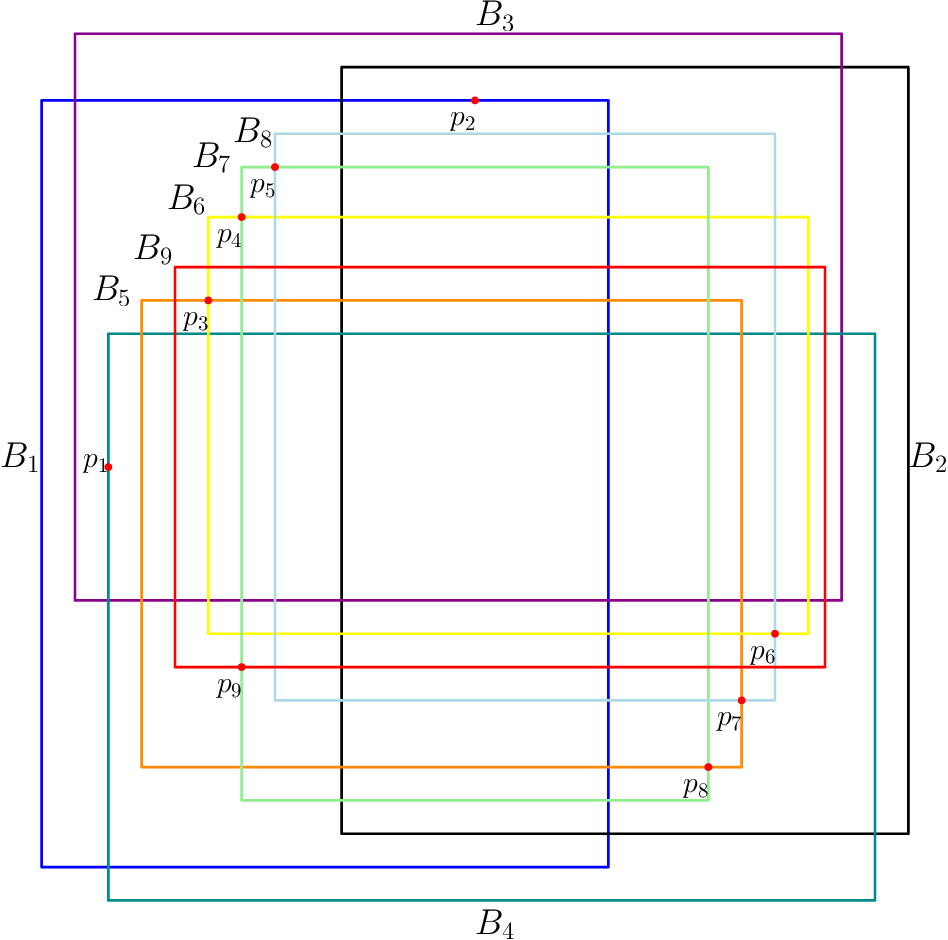}
    
    \caption{Example showing $N(2,2,2)\geq 3$}
    \label{fi5}
\end{figure}

Let us show that $N(2,2,2)\geq 3$, that is, there are a family of boxes and a finite set of points in the plane such that any two boxes share a common point in the point set and the family of boxes cannot be pierced by 2 points of this set.

Figure~\ref{fi5} illustrates a family $\mathcal B=\{B_1,\dots, B_9\}$ of 9 boxes and a set $P=\{p_1,\dots, p_9\}$ of 9 points in the plane. Any two boxes share a common point. In order to pierce $B_1,\dots, B_8$ by $2$ points, we must choose $p_4$ and $p_7$ for that, however the box $B_9$ is not pierced by those two points. The piercing number is at least 3.

By lifting this example to higher dimensional spaces, one can easily show that $N(2,2,d)\geq 3$ for $d\geq 2$.

\end{example}

\subsection{Using the $(p,q)$-type Halman theorem}

\begin{corollary}[corollary of Theorem~\ref{pq Halman theorem for q at least 2}]
\label{corollary of pq Halman theorem for q at least 2}
    Let $k$ and $d$ be positive integers with $2d-k\geq 2$. Let  $P\subset \mathbb R^d$ be a finite set, and let $\mathcal B$ be a finite family of boxes in $\mathbb R^d$. If for every subfamily $\mathcal B'\subseteq \mathcal B$ of size at most $2d-k$, its trace $\mathcal B'|_P$ is intersecting, then we have 
    \[
 \tau (\mathcal B|_P)\leq \begin{cases} N(3,3,1+\frac{k+1}{2}) & \text{ if $k$ is odd},  \\ N(2,2, 1+\frac{k}{2}) & \text{ if $k$ is even}. \end{cases} \]
\end{corollary}

\begin{proof}
Our proof is based on Lemma~\ref{lemma induction reduction baby version}. To apply it, we set $s=2d-k$, $t=\lfloor \frac{2d-k-2}{2}\rfloor$, and $\mathcal B_1 = \dots = \mathcal B_{s} =\mathcal B$. Note that 
\[
s-2t=\begin{cases}
    2 \text{\quad if $k$ is even;}\\
    3 \text{\quad if $k$ is odd}
\end{cases}
\quad \text{ and }\quad
d-t=\begin{cases}
    1+\frac{k}{2}& \text{\quad if $k$ is even;}\\
    1+\frac{k+1}{2}& \text{\quad if $k$ is odd}.
\end{cases}
\]
By applying Lemma~\ref{lemma induction reduction baby version}, we have that there are a finite set $P'\subset \mathbb R^{d-t}$ and a family $\mathcal B'$ of boxes in $\mathbb R^{d-t}$ such that
\[
\tau(\mathcal B|_P)\leq \tau(\mathcal B'|_{P'})\leq N(s-2t,s-2t,d-t),
\]
where the last inequality holds by Theorem~\ref{pq Halman theorem for q at least 2}.
\end{proof}

\section{Discussion}\label{section: discussion}
We studied the quantitative colorful Helly problem for $2d-1$ families and obtained a nearly tight result. A natural direction for further study is to determine whether our lower bound is indeed tight.

\begin{problem}
    Let \( n \geq 1 \) and \( d \geq 2 \) be integers. Let $P\subset \mathbb R^d$ be a finite set, and let \( \mathcal{B}_1, \dots, \mathcal{B}_{2d-1} \) be finite families of boxes in $\mathbb R^d$. Prove or disprove that if for any their transversal $\mathcal B$, its trace $\mathcal B|_P$ is \( n \)-intersecting, then there exists $j\in [2d-1]$ such that $\tau_n'(\mathcal B_j|_P)\leq 2n$.
\end{problem}

Moreover, we believe that the approach developed to prove Theorem~\ref{quantitative colorful Halman 2d-1 colors} can be adapted to establish upper bounds when $d+1 \leq k \leq 2d-2$ families of boxes are considered instead of $2d-1$. As a concrete open problem, we propose the following question.

\begin{problem}
Let $P \subset \mathbb{R}^2$ be a finite set, and let $\mathcal{B}$ be a family of boxes. Prove or disprove that if, for every subfamily $\mathcal{B}' \subseteq \mathcal{B}$ of size at most $2$, the trace $\mathcal{B}'|_P$ is intersecting, then $\tau(\mathcal{B}|_P) \leq 3$.
\end{problem}

\subsection*{Acknowledgements}
AP is supported by the NSF grant DMS 2349045. 

The authors thank Alexander Golovanov for carefully proofreading the proof of Theorem~\ref{quantitative colorful Halman 2d-1 colors}, and M\'arton Nasz\'odi for bringing the paper~\cite{furedi1988matchings} to their attention.

\bibliographystyle{alpha}
\bibliography{references}

\end{document}